\documentclass[11pt]{amsart} 
\usepackage{amssymb, amsthm, epsfig, graphicx, psfrag, verbatim,
  amscd, enumerate, stmaryrd, amsmath}

\newtheorem{thm}{Theorem}[section]

\newtheorem{lem}[thm]{Lemma}
\newtheorem{prop}[thm]{Proposition}

\theoremstyle{definition}
\newtheorem{defn}[thm]{Definition}

\newtheorem{rem}[thm]{Remark}

\theoremstyle{remark}

\numberwithin{equation}{section}

\newcommand{\ep}{\varepsilon}

\newcommand{\la}{\lambda}

\newcommand{\si}{\sigma}
\newcommand{\Si}{\Sigma}

\newcommand{\csi}{\xi}

\newcommand{\bfz}{{\mathbb {Z}}}
\newcommand{\bfn}{{\mathbb {N}}}

\newcommand{\+}{\oplus}

\newcommand{\Z}{\mathbb Z}
\newcommand{\N}{\mathbb N}

\newcommand{\R}{\mathbb R}

\newcommand{\del}{\partial}

\newcommand{\co}{\colon\thinspace} 
\newcommand{\CP}{{\mathbb C}{\mathbb P}}

\newcommand{\bCPk}{{\overline {{\mathbb C}{\mathbb P}^2}}}

\newcommand{\sg}{\operatorname{sg}}

\renewcommand{\int}{\operatorname{int}}

\newcommand{\sgn}{\operatorname{sgn}}

\begin{document}
\mathsurround=1pt 
\title[Singular maps on exotic $4$-manifold pairs]
{Singular maps on exotic $4$-manifold pairs}



\author{Boldizs\'ar Kalm\'ar}
\address{Alfr\'ed R\'enyi Institute of Mathematics,
Hungarian Academy of Sciences \newline
Re\'altanoda u. 13-15, 1053 Budapest, Hungary}
\email{bkalmar@renyi.hu}

\author{Andr\'as I. Stipsicz}
\address{Alfr\'ed R\'enyi Institute of Mathematics,
Hungarian Academy of Sciences \newline
Re\'altanoda u. 13-15, 1053 Budapest, Hungary and \newline
Institute for Advanced Study, Princeton, NJ}
\email{stipsicz@renyi.hu}

\thanks{\today}

\begin{abstract}
  We show examples of  pairs of smooth, compact, homeomorphic $4$-manifolds,
  whose diffeomorphism types are distinguished by the topology of the
  singular sets of  smooth stable maps defined on them. In this distinction we rely on 
  results from Seiberg-Witten theory.
\end{abstract}

\maketitle

\section{Introduction}
Different smooth structures on a given topological $4$-manifold have
been shown to exist by a rather delicate count of solutions of certain
geometric PDE's associated to the smooth structure (and some further
choices, such as a Riemannian metric and possibly a spin$^c$
structure). This idea was the basis of the definition of Donaldson's
polynomial invariant \cite{donaldson}, as well as the Seiberg-Witten
invariants \cite{witten}. In these invariants specific connections
(and sections of bundles associated to the further structures on the
$4$-manifold) have been counted for the potentially different smooth
structures.  By the pinoneering work of Kronheimer and Mrowka
\cite{KrMr} it was clear that, through the \emph{adjunction
  inequalities}, the invariants provide strong restrictions on the
topology of surfaces smoothly embedded in the $4$-manifold
representing some fixed homology classes. In a slightly different
direction, work of Taubes \cite{Taubes} provided obstructions for the
existence of symplectic structures compatible with the chosen smooth
structure in terms of the Seiberg-Witten invariants.  This idea then
leads to a simple distinction between certain pairs of smooth
structures: one of them (which is compatible with a symplectic
structure) admits a Lefschetz fibration (or more generally a Lefschetz
pencil) map \cite{Donaldsonpencil}, while the one which is not
compatible with any symplectic structure does not admit such a map
\cite{Gompf, GS99}.  Similarly, for $4$-manifolds with nonempty
boundary there are topological examples with two smooth structures
such that one admits a Lefschetz fibration with bounded fibers (and
hence a Stein structure), cf.\ Loi-Piergallini \cite{LP} and
Akbulut-Ozbagci \cite{AO}, while the other smooth structure does not
carry a Stein structure, and therefore does not carry a Lefschetz
fibration map either. Such pair of smooth structures was found by
Akbulut-Matveyev~\cite{AM}, cf.\ Theorem~\ref{thm:am}.

Further notable examples of manifolds distinguished by some
properties of smooth maps defined on them are provided by ``large''
exotic $\R^4$'s, since these noncompact $4$-manifolds do not admit
embedding into the standard Euclidean $4$-space $\R^4$, cf.\
\cite[Section~9]{GS99}.  Examples of similar kind are the smooth
structures on certain connected sums of $S^2$-bundles over surfaces
(cf.\ the exotic structures described by J. Park \cite{Pa00}), which
have submersions with definite folds only for the standard structure
\cite{SS99}.

In the present work we will find properties of stable/fold maps such
that the geometry of their singular sets will distinguish exotic
smooth structures on some appropriately chosen topological
$4$-manifolds. We will apply a result of Saeki from \cite{Sa03}
(cf.\ also Theorem~\ref{thm:saeki}) in constructing maps with the
desired properties on some of our examples. We appeal to
Seiberg-Witten theory (in particular, to the adjunction inequality and
its consequences) in showing that maps with certain prescribed
singular sets do not exist on our carefully chosen other examples. The
first and most obvious pair of examples for such phenomena is provided
by Akbulut-Matveyev \cite{AM} (cf.\ Theorem~\ref{thm:am}) --- in the
following we extend their idea to an infinite family of such exotic
pairs (given in Theorem~\ref{thm:vegtelen}).

To start our discussion, suppose that $X$ is a given smooth
$4$-manifold.  For a smooth manifold $Y$ the smooth map $f \co X \to
Y$ is called {\it stable} if for every smooth map $g$ sufficiently
close to $f$ in $C^{\infty}(X,Y)$ there are diffeomorphisms $D_X \co X
\to X$ and $D_Y \co Y \to Y$ such that $D_Y \circ f = g \circ D_X$.
Considering the special case $Y=\R ^3$, stable maps are dense in
$C^{\infty}(X,\R^3)$ and the singular set of a stable map $f \co X \to
\R^3$ is an embedded (possibly non-orientable) surface $\Si_f \subset X$.  A stable map is
called a \emph{fold} map if it has only fold singularities.  Indeed,
for a stable map $f \co X \to \R^3$ a point $p \in X$ is a
\emph{fold singularity} if $f$ can be written in some local charts
around $p$ and $f(p)$ as $(x, y, z, v) \mapsto (x^2 \pm y^2, z, v)$.
A fold singularity with ``$+$'' in this formula is called 
a \emph{definite fold singularity} and
with ``$-$'' it is called an \emph{indefinite fold singularity}.
If $X$ is a smooth $4$-manifold with nonempty boundary, then
by a stable map of $X$ into $Y$ we mean a map of $X$ into $Y$ which can be extended to
a stable map $f \co \tilde X \to Y$, where $\tilde X$ is a smooth $4$-manifold without boundary,
$X \subset \tilde X$ is a smooth submanifold and 
$f$ has no singularities in a neighborhood of $\del X$.

Let $\mathcal M \subset C^{\infty}(X,\R^3)$ be a fixed subset of stable maps
with singular set consisting only of closed orientable surfaces.  In the
following we will define a smooth invariant of $X$ denoted by $\sg_{\mathcal
  M}(X)$ in terms of the possible genera of the components of the singular
sets of the maps in $\mathcal M$.
More formally, fix an integer $k \geq 1$, take a map $f \in \mathcal M$ and
write the singular set of $f$ in the form $\bigcup_{i=1}^n \Si_f^i$, where
$\Si_f^i \subset X$ are the connected components.  For each $k$-element subset
$I=\{ i_1, \ldots, i_k \}$ of $\{1, \ldots, n\}$ denote by $g_{ I}(f)\in {\mathbb
  {N}}$ the \emph{maximal} genus of the surfaces $\Si_f^{i_1}, \ldots,
\Si_f^{i_k}$.  Define the set $G_{\max}^k(f)$ as
\[
G_{\max}^k(f)=
\{ g_{I} (f) \mid I\subset \{ 1, \ldots , n \} \text{ and } \vert I\vert=k\}.
\]
Finally, we define $\sg^k_{\mathcal M}(X)$ as
\[
\min \bigcup_{f \in \mathcal M} G_{\max}^k(f).
\]
This is a non-negative integer or it is equal to $\infty$ if the  set
$\bigcup_{f \in \mathcal M} G_{\max}^k(f)$ 
is empty.  In the next definition 
we consider only those components of the singular set of a
map in $\mathcal M$, which represent a fixed homology class and consist only
of a fixed set of singularity types. This leads us to

\begin{defn}\label{sginvdef}
  For a smooth $4$-manifold $X$ (possibly with nonempty boundary) let $A
  \subseteq H_2(X; \Z)$ be a set of second homology classes and let $\mathcal
  S$ be a fixed set of singularity types.  For a stable map $f \in \mathcal M$
  let $\bigcup _{i=1}^n \Si _f^i$ denote the union of those components of the
  singular set of $f$ which have the property that $\Si _f^i$ represents a
  homology class in $A$ and contains singularities only from $\mathcal S$.
  As before, for a fixed integer $k \geq 1$ and $I\subset \{ 1, \ldots , n\}$
  with $\vert I \vert =k$ let ${g}_{I}(f,A)$ denote the maximal genus of $\Si
  _f^i$ with $i\in I$.  Then $G^k_{\max }(f,A)$ is the set of all $g_I(f,A)$
  (where $I$ runs through the $k$-element subsets of $\{ 1, \ldots , n\}$),
  and $\sg^{k}(X, A) =\sg^{k}_{\mathcal M, \mathcal S}(X, A)$ is the minimum
  of the union $\bigcup_f {G}^{k}_{\max}(f, A)$, where $f$ runs over the
  stable maps in $\mathcal M$.
\end{defn}
\begin{rem}\noindent
\begin{enumerate}
\item
For any $k \geq 1$ we have $\sg^k(X, A) \in \Z_{\geq 0} \cup \{ \infty \}$ since
the minimum of the empty set is defined to be $\infty$.  
\item The reason for the slightly complicated definition of the invariant $\sg
  ^k$ is that in our applications we will find 4-manifold pairs with the
  property that in one $4$-manifold $k$ disjoint $(-1)$-spheres can be
  located, while in the other one we show that there are no $k$ disjoint
  $(-1)$-spheres. 
  By fixing the appropriate
  homology classes for $A$, the invariant $\sg ^k (X,A)$ will distingush these
  $4$-manifolds, cf. Theorems~\ref{generatoros} and \ref{minuszegyes}.
\item
In the case of $k=1$ the value of $\sg^1(X, A)$ is just the minimum of all
the genera of the possible allowed singular set components of the maps in $\mathcal M$.
\item
A fairly natural invariant of the same spirit is
taking  $\min_{f \in \mathcal M} \{ g(f) \}$, where
${g}(f)$ denotes the maximal genus of {\it all} the components of $\Si$ in Definition~\ref{sginvdef}.
We denote this by $\sg(X, A)$ and a slightly different version of
it will play a role  in Definition~\ref{sgpropinvdef} and Theorem~\ref{kanonikus}.
\item
For  $1 \leq k \leq l$, 
we have $\sg^k(X, A) \leq \sg^{l}(X, A)$ as it can be seen easily from the definition.
\end{enumerate}
\end{rem}

In the present paper we make two main choices for $\mathcal M$ and
$\mathcal S$:
\begin{enumerate}[(i)]
\item
let $\mathcal M$ be the set of all the stable maps with singular set
consisting only of closed orientable surfaces and let $\mathcal S$ be
the set of all the singularities, or
\item
let $\mathcal M$ be the set of all the fold maps with singular set
consisting only of closed orientable surfaces and let $\mathcal S$ be
the one element set of the definite fold singularity.
\end{enumerate}

Our results work with both choices. By
constructing specific stable maps, and estimating the invariant $\sg
^k (X, A)$ for particular manifolds and homology classes using
Seiberg-Witten theory, we prove

\begin{thm}\label{generatoros}
There exist homeomorphic smooth $4$-manifolds $X_1$ and $X_2$ with
$H_2 (X_1; \bfz )\cong H_2 (X_2; \bfz ) \cong \bfz$ such that for 
the
$2$-element set of generators $A\subset
H_2 (X_1; \bfz )$ we have
\begin{enumerate}[\rm (1)]
\item
$
\sg^{1}(X_1,  A) = 0 \mbox{ \ and \ } 0 < \sg^{1}(X_2, A) < \infty
$
and
\item
$\sg^{k}(X_1,  A)  = \sg^{k}(X_2,  A)  = \infty$ for $k \geq 2$.
\end{enumerate}
\end{thm}
\begin{rem}
It follows easily from the proof of Theorem~\ref{generatoros} that
also $\sg(X_1,  A) = 0$ and $0 < \sg(X_2, A) < \infty$.
\end{rem}
Our example for the topological manifold $X_1$ in the above theorem is
with nonempty boundary.  For closed manifolds we show the following
result:

\begin{thm}\label{minuszegyes}
There exist homeomorphic, smooth, closed
$4$-manifolds $V$ and $W$
and there exists  $1 \leq k \leq 4$ such that 
if $A$ denotes
the set of homology classes in $H_2(V) \cong H_2(W)$ having
self-intersection $-1$, then
\begin{enumerate}[\rm (1)]
\item
$\sg^{k}(V, A) = 0$ and $0 < \sg^k(W, A) < \infty$,  and
\item
$\sg^{l}(V, A) = \sg^{l}(W, A) = 0$ for any $1 \leq l < k$.
\end{enumerate}
\end{thm}

Similar results can be derived by looking at the genera of singular
sets of smooth maps satisfying some property related to the boundary
$3$-manifold of the source. To state this result, we need the
following definition.

\begin{defn}\label{sgpropinvdef}
  A stable map $f \co X \to \R^3$ with $\del X \neq \emptyset$ 
   induces a stable framing $\phi$
  on $\del X$.  Let $\mathcal A$ be the property that the stable
  framing induced on the boundary of the source is canonical, i.e.,
  has minimal total defect in the sense of \cite{KiMe}. (For
  discussions of these notions, see also Section~\ref{sec:defect}.)
  For a $4$-manifold $X$ and a stable map $f \co X \to \R^3$ with
  singular set  consisting only of closed orientable surfaces
  let ${\rm g}(f)$ denote the maximal genus of the components of
  the {definite fold} singular set.  Let $\sg(X, \mathcal A)$ denote the minimum of the values
  ${\rm g}(f)$, where $f$ runs over the {fold} maps of $X$
  into $\R^3$ with property $\mathcal A$ and with singular set
  consisting only of closed orientable surfaces.
\end{defn}

With this notion at hand, now we can state our next result:

\begin{thm}\label{kanonikus}
There exist homeomorphic smooth compact $4$-manifolds $X_1$ and $X_2$ such that
$\sg(X_1, \mathcal A) = 0$ and $0 < \sg(X_2, \mathcal A) < \infty$.
\end{thm}

The prominent example of a pair $(X_1,X_2)$ of smooth $4$-manifolds
(with nonempty boundary) used in the above results was found by
Akbulut and Matveyev \cite{AM}. In order to show that our method is
applicable in further examples, we extend the examples of \cite{AM}:
\begin{thm}\label{thm:vegtelen}
There is an infinite family $(X_1(n), X_2(n))_{n\in \bfn}$ of
homeomorphic, non-diffeomorphic compact $4$-manifold-pairs which are
non-homeomorphic for different $n$'s and Theorem~\ref{generatoros}
distinguishes the smooth structure of $X_1(n)$ from the smooth
structure of $X_2 (n)$.
\end{thm}

The paper is organized as follows. In Section~\ref{sec:two} we show
infintely many examples for which Theorems~\ref{generatoros},
\ref{minuszegyes} and \ref{kanonikus} apply, and hence provide a proof
of Theorem~\ref{thm:vegtelen}. (A tedious computation within the proof
of this theorem is deferred to an Appendix in
Section~\ref{sec:appendix}.)  In Sections~\ref{sec:three} and
\ref{sec:defect} we provide the proofs of the results described
in the Introduction.

\bigskip

{\bf Acknowledgements:} The authors were supported by OTKA NK81203 and
by the \emph{Lend\"ulet program} of the Hungarian Academy of Sciences.
The first author was partially supported by Magyary Zolt\'an
Postdoctoral Fellowship. The second author was partially supported by the ERC
Grant LDTBud. He also wishes to thank the Institute for Advanced Study
for the stimulating research environment while part of this work has been
performed.

\section{An infinite family of exotic $4$-manifold pairs}
\label{sec:two}

To make the proofs of Theorems~\ref{generatoros}, \ref{minuszegyes}
and \ref{kanonikus} of the Introduction more transparent, we start by
describing the examples promised in Theorem~\ref{thm:vegtelen}.  The
following idea of constructing exotic pairs of $4$-manifolds is due to
Akbulut-Matveyev \cite{AM}. Suppose that $K_1, K_2$ are two given
knots in $S^3$ with the following properties:

\begin{itemize}
\item the $3$-manifold given by $(-1)$-surgery on $K_1$
is diffeomorphic to the $3$-manifold given by $(-1)$-surgery
on $K_2$,
\item $K_1$ is slice, that is, bounds a properly, smoothly embedded
disk in $D^4$, and
\item the maximal Thurston-Bennequin number of $K_2$ is nonnegative,
in particular, there is a Legendrian knot $L$ (in the standard 
contact $S^3$) which is smoothly isotopic to $K_2$ and tb$(L)=0$.
(For the definition of the Thurston-Bennequin number, see for example
\cite{OS}.)
\end{itemize}

An example for such a pair $(K_1, K_2)$ was found by Akbulut and
Matveyev in \cite{AM} (cf.\ also \cite[Theorem~11.4.8]{GS99}).  In the
following, $X_i$ will denote the smooth $4$-manifold obtained by
attaching a $4$-dimensional $2$-handle to $D^4$ along $K_i$ with
framing $-1$ for $i=1,2$.  As it was shown in \cite{AM}, the
properties listed above allow us to prove that
\begin{thm}\label{thm:am}(\cite{AM})
The smooth $4$-manifolds $X_1, X_2$ are homeomorphic but not diffeomorphic.
\end{thm}
For the convenience of the reader, we include a short outline of the
proof of this theorem.
\begin{proof}
Since both $X_1$ and $X_2$ are given as a single $2$-handle attachment
to $D^4$, both are simply connected. Since the surgery coefficient
fixed on $K_j$ in both cases is $(-1)$, the boundaries $\partial X_1$
and $\partial X_2$ (which are assumed to be diffeomorphic) are
integral homology spheres. Furthermore the intersection forms
$Q_{X_1}$ and $Q_{X_2}$ can be represented by the $1\times 1$ matrix
$\langle -1\rangle$, in particular, are isomorphic. The extension of
Freedman's fundamental result on topological $4$-manifold to the case of
$4$-manifolds with integral homology sphere boundary (and trivial
fundamental group) \cite{FQ} then implies that $X_1$ and $X_2$ are
homeomorphic.

Since $K_1$ is a slice knot, the generator of $H_2(X_1; \bfz )$ can be
represented by an embedded sphere. On the other hand, since $K_2$ is
smoothly isotopic to a Legendrian knot with Thurston-Bennequin number
tb=0, the famous theorem of Eliashberg \cite{Eli} (cf.\ also \cite{CE})
implies that $X_2$ admits a Stein structure. Since Stein manifolds
embed into minimal surfaces of general type \cite{LM}, and a minimal
surface of general type does not contain a smoothly embedded sphere
with homological square $-1$, we conclude that the generator of $H_2(
X_2; \bfz )$ cannot be represented by a smoothly embedded sphere.
(The statement about minimal surfaces of general type relies on a
computational fact in Seiberg-Witten theory: minimal surfaces of
general type have two basic classes $\pm c_1\in H^2$, but $c_1^2>0$ for those
surrfaces.)  This conclusion, however, shows that $X_1$ and $X_2$ are
non-diffeomorphic.
\end{proof}

The example of Akbulut-Matveyev can be generalized to an
infinite sequence of pairs of knots which we describe presently.
\begin{figure}[ht] 
\begin{center} 
\epsfig{file=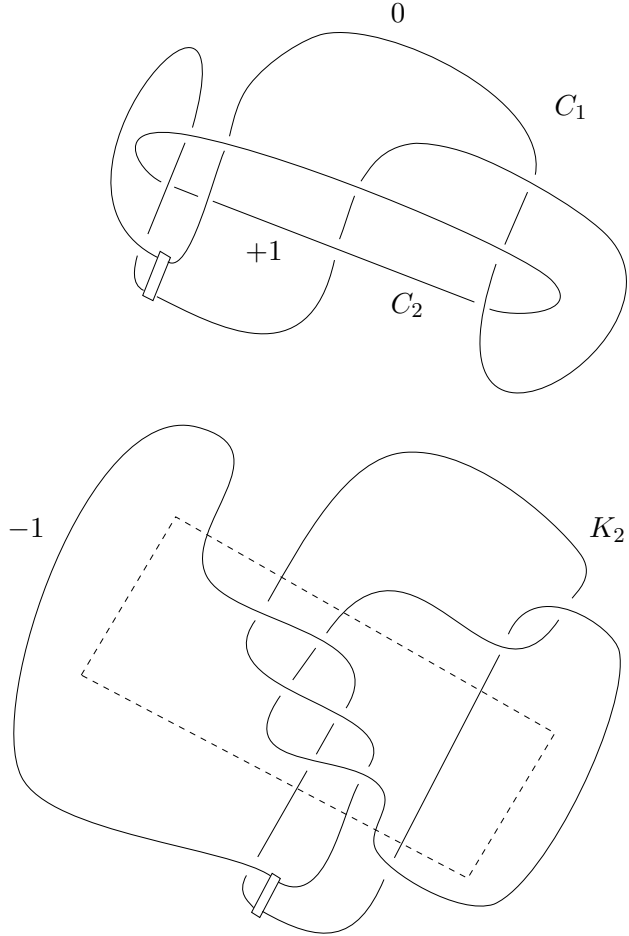, height=12cm} 
\put(-235, 150){$-1$}
\put(-15, 150){$K_2$}
\put(-90, 235){$C_2$}
\put(-28, 310){$C_1$}
\put(-90, 345){$0$}
\put(-145, 255){$+1$}
\end{center} 
\caption{{\bf The framed link $(C_1, C_2)$.}  In the upper diagram $C_1$
  has framing $0$ while the unknot $C_2$ has framing $+1$.  The small
  box in the upper figure represents a tangle which will be specified
  later. The lower diagram shows the result of the blow-down of the
  $(+1)$-framed unknot $C_2$, and hence the framing of $K_2$ is equal to
  $(-1)$.  The dashed box shows the location of the blow-down.}
\label{elsolink_lefujas}  
\end{figure}
Consider the $2$-component link $(C_1, C_2)$ given by the upper diagram
of Figure~\ref{elsolink_lefujas}.  The small box intersecting $C_1$ in
the lower left corner of the diagram contains a tangle which will be
specified later.  Equip $C_1$ with framing $0$ and $C_2$ with framing
$+1$. If we blow down the unknot $C_2$, we get a knot $K_2$, depicted by
the lower picture of Figure~\ref{elsolink_lefujas}.  It is easy to see
that the framing of $K_2$ is equal to $-1$. By isotoping the result
slightly we get the front projection of a Legendrian knot isotopic to
$K_2$ as in Figure~\ref{csomo}.
\begin{figure}[ht] 
\begin{center} 
\epsfig{file=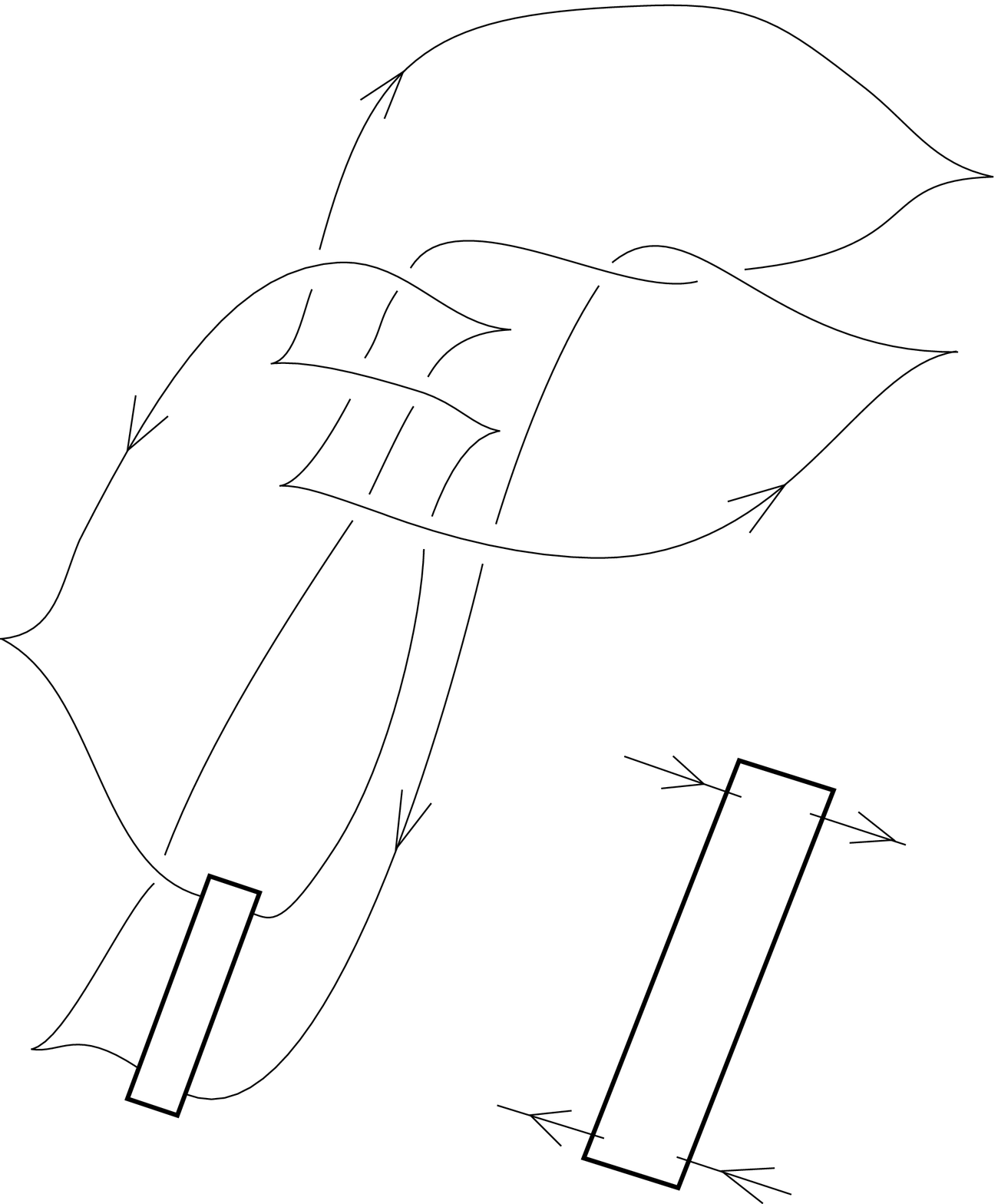, height=11cm} 
\end{center} 
\caption{{\bf The knot $K_2$ can be isotoped to be the front
    projection of a Legendrian knot.} Indeed, by considering an
  appropriate modul in the box, the resulting Legendrian knot will
  have vanishing Thurston-Bennequin number.}
\label{csomo}  
\end{figure} 
Simple computation of the writhe and the number of cusps shows that
for any module in the small box with non-negative tb, the resulting
knot $K_2$ will have non-negative tb. In particular, if we use the
module shown by Figure~\ref{modulpelda2} then we get a Legendrian
\begin{figure}[ht] 
\begin{center} 
\epsfig{file=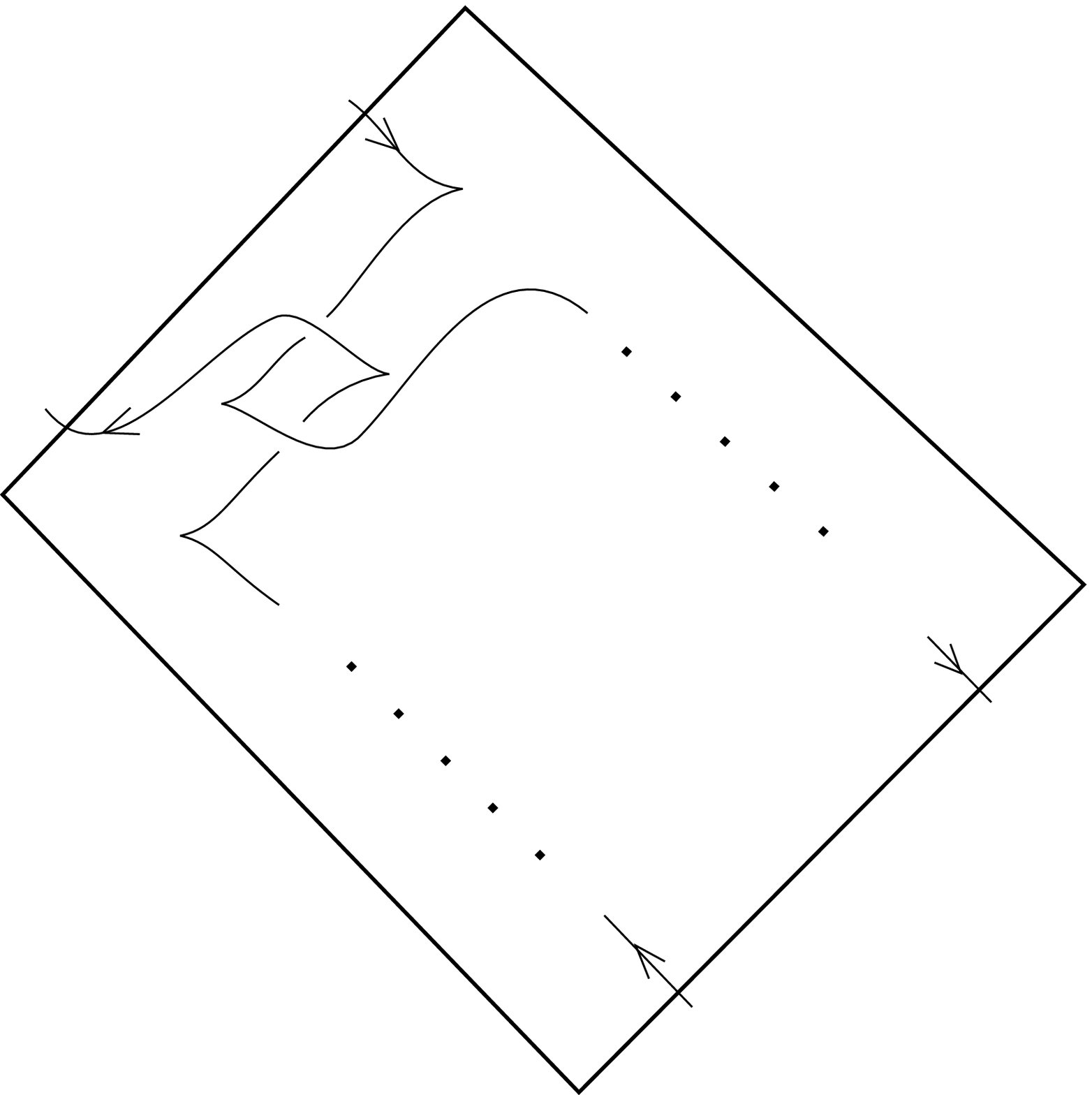, height=8cm} 
\end{center} 
\caption{{\bf The module with $n$ full left twists provides knots
    $K_2(n)$ with vanishing Thurston-Bennequin numbers.} In the
  diagram the module is already in Legendrian position. Obviously, by
  addig a left- and a right-cusp to the diagram we get a Legendrian
  unknot with Thurston-Bennequin number equal to $-1$.}
\label{modulpelda2}  
\end{figure} 
realization of $K_2$ with vanishing Thurston-Bennequin number.  If we
insert the module of Figure~\ref{modulpelda2} with $n$ full left
twists into the box of Figure~\ref{csomo}, the resulting knot will be
denoted by $K_2(n)$. Notice that with this choice of the module, $C_1$
is an unknot. By isotoping $C_1$ together with the unknot $C_2$ (as it
is shown by Figure~\ref{linkekizotopok})
\begin{figure}[ht] 
\begin{center} 
\epsfig{file=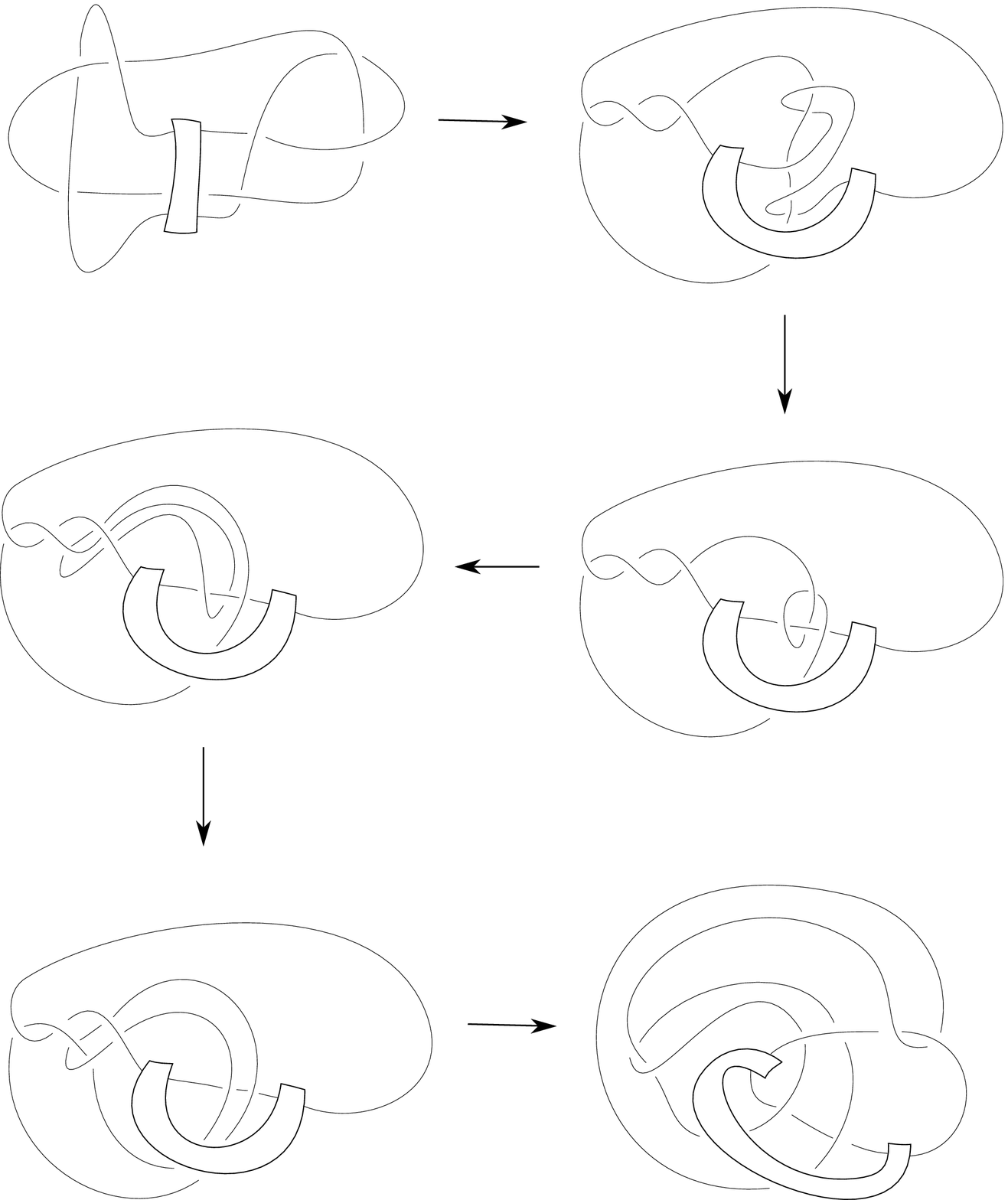, height=14cm} 
\end{center} 
\caption{{\bf The isotopy on the link of
    Figure~\ref{elsolink_lefujas}.} The diagrams show the intermediate
  stages of the isotopy transforming the link of
  Figure~\ref{elsolink_lefujas} into the link of
  Figure~\ref{masiklink}.}
\label{linkekizotopok}  
\end{figure} 
and then surgering out the $0$-framed unknot $C_1$ to a $1$-handle, we get
Figure~\ref{masiklink}.
\begin{figure}[ht] 
\begin{center} 
\epsfig{file=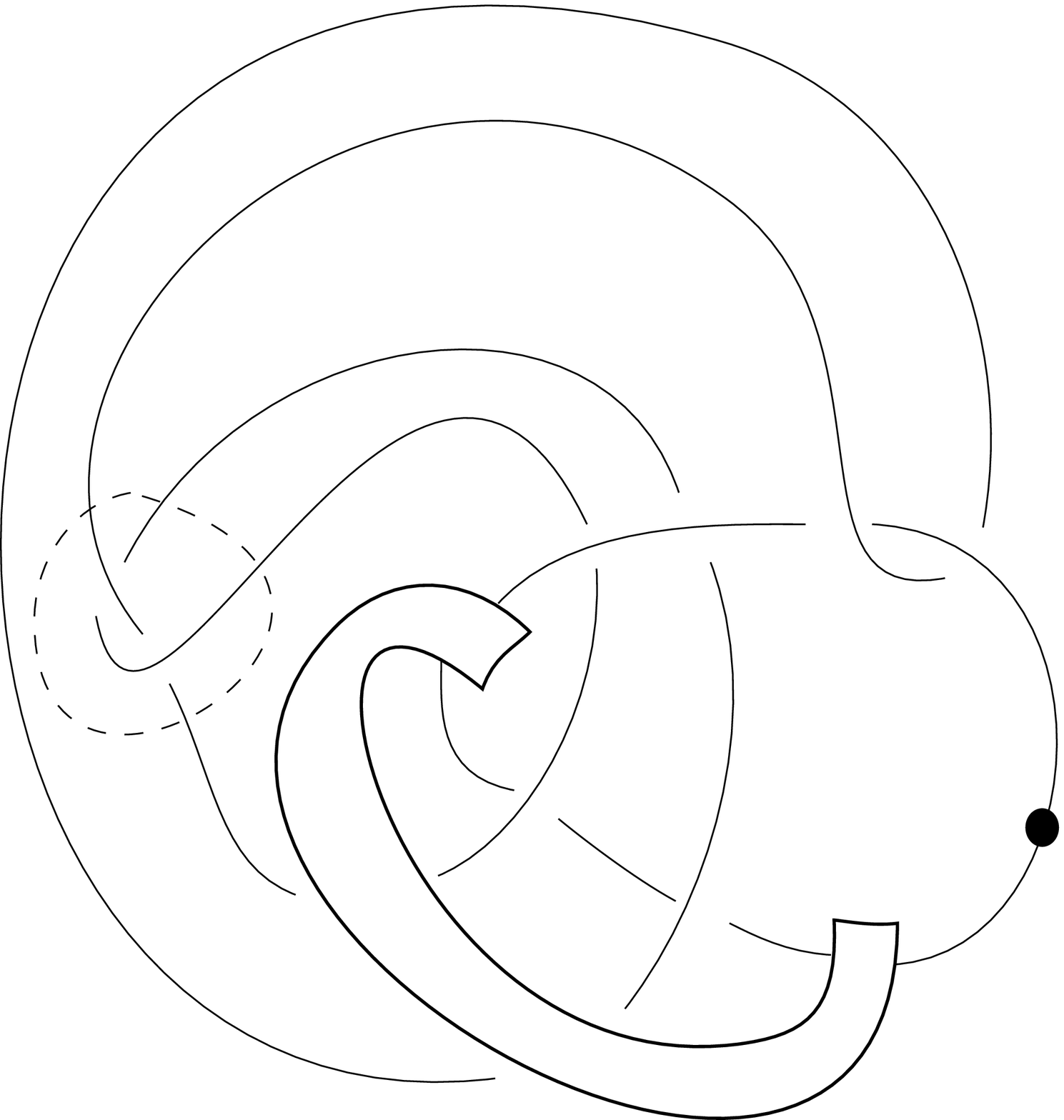, height=7cm} 
\put(-30, 185){$C_2$}
\end{center} 
\caption{ After blowing up the clasp in the dashed circle, the two
  handles will form a cancelling pair, showing that the unknot
  introduced by the blow-up gives a slice knot in $S^3$.}
\label{masiklink}  
\end{figure} 
By blowing up $C_2$ at the dashed circle on the picture, and then
pulling the resulting $(-1)$-framed unknot through the $1$-handle, we
get a slice knot $S$ with framing $(-1)$. Indeed, after the blow-up,
the $1$-handle resulting from the surgery along $C_1$ and the
$2$-handle attached along $C_2$ can be isotoped to have geometric
linking one, hence these handles form a cancelling pair. (We do not
draw an explicite diagram of the slice knot $S$ here, cf.\ also
Remark~\ref{rem:kesobb}.)  Now take $K_1$ to be equal to this knot
$S$; more precisely, if we use the module of Figure~\ref{modulpelda2}
(with $n$ full left
twists in the module)
in the small box of Figure~\ref{masiklink}, then denote the resulting knot by $K_1(n)$.  Denote
the $4$-manifold we get by attaching a $2$-handle to $D^4$ along
$K_j(n)$ by $X_j(n) $ ($j\in \{ 1,2\}, n\in \bfn $).  Notice that
since we only blew up and down, isotoped and surgered a $0$-framed
$2$-handle, the fact that the boundaries of the $4$-manifolds $X_1(n)$
and $X_2(n)$ are diffeomorphic $3$-manifolds is a simple fact (since
both are diffeomorphic to the boundary of the $4$-manifold we get from
Figure~\ref{elsolink_lefujas}).  The original example of Akbulut and
Matveyev is the pair $(K_1(n), K_2(n))$ for $n=0$.  The same argument
as the one presented in the proof of Theorem~\ref{thm:am} shows that
$X_1(n)$ and $X_2(n)$ are homeomorphic but not diffeomorphic
$4$-manifolds.
\begin{rem} \label{rem:kesobb}
It is somewhat tricky (but not difficult) to see that the knot $S$ we
chose for $K_1(n) $ is actually a slice knot for every $n\in \bfn $.
Notice that, strictly speaking, we do not need this fact in our
argument when showing that the $4$-manifolds $X_1(n), X_2(n)$ are
nondiffeomorphic. The fact that the generator of $H_2(X_1(n); \bfz )$ can
be represented by a sphere easily follows from the description of
$X_1(n)$ we get after blowing up the clasp in the dashed circle in
Figure~\ref{masiklink} --- the exceptional sphere of the blow-up will
represent the generator of the second homology, since the two other
handles form a cancelling pair in homology.
\end{rem}   

In fact, the same line of reasoning applies for all knot pairs we get by
putting an appropriate module into the box of Figure~\ref{elsolink_lefujas}.
In general, we have
\begin{thm}
If the module contains an oriented Legendrian diagram of the unknot
after removing a left-most and a right-most cusps (with appropriately  orientated arcs)
with ${\mathrm {tb}} = 0$, then 
\begin{enumerate}[\rm (1)]
\item
$K_0$ is the unknot,
\item
$K_1$ is slice,
\item
$K_2$ has a Legendrian realization with ${\mathrm {tb}} = 0$, and
\item
the $4$-manifolds $X_1$ and $X_2$ are homeomorphic but not diffeomorphic. \qed
\end{enumerate}
\end{thm}

It is unlcear, however, whether in general the resulting exotic pairs will provide
new examples.  In order to prove that the pairs $(X_1(n), X_2(n))$ 
for the particular modules of Figure~\ref{modulpelda2}   do provide an infinite
sequence of extensions of the example of \cite{AM}, it is enough to
show that the boundary integral homology spheres are different for
 $n \geq 0$.

\begin{prop}\label{prop:veg}
The integral homology spheres obtained by 
$-1$ surgery along $K_2(n)$ are pairwise not homeomorphic
for  $n \geq 0$. 
\end{prop}
\begin{proof}
The \emph{Ohtsuki invariants} $\la_k (Y)$ 
of an integral homology sphere $Y$
(extracted from the quantum invariant of 
Reshetikin-Turaev of the $3$-manifold) can be used to
distinguish integral homology spheres.
These invariants were introduced by Ohtsuki
 \cite{Oh95}, and a more computable derivation of the invariants was 
given in \cite{LW99}.  $\la_1(Y)$ is
determined by the Casson invariant of the $3$-manifold $Y$, and in case
$Y$ is given as integral surgery along a knot $K$, the invariant $\la_2 (Y) $ can be
computed from the Jones polynomial and the Conway polynomial of
$K$. More precisely, if $Y$ is given as $(-1)$-surgery along a knot
$K\subset S^3$ then by \cite[Theorem~5.2]{LW99}
\[
\la_2 (Y)= \frac{v_2(K)}{2} + \frac{v_3(K)}{3} + \frac{5}{3}v_2^2(K) - 60c_4(K),
\]
where $c_4(K)$ is the coefficient of $z^4$ in the Conway polynomial of
$K$ and $v_i (K)= \frac{ \del^i V(K, e^{h})}{\del h^i}(0)$, where
$V(K,t)$ is the Jones polynomial of $K$.  (The Conway and Jones
polynomials are defined by skein theory and normalization as given in
\cite{LW99}.) A somewhat tedious computation (postponed to an
Appendix, cf.\ Lemmas~\ref{lem:conwayszam} and \ref{lem:jonesszamok})
shows that for the knot $K_2(n)$ the value of $c_4$ is $-n$, the value
of $v_2$ is $-12$, while $v_3$ is $36n +108$.  This shows that $\la_2$
of the $3$-manifold $S^3_{-1}(K_2(n))$ we get by $(-1)$-surgery along
$K_2(n)$ is equal to $72n +270$.  Therefore the Ohtsuki invariants
$\lambda _2$ of the manifolds $S^3_{-1}(K_2(n))$ are all different,
implying that the 3-manifolds are pairwise nondiffeomorphic.
\end{proof}

\begin{proof}[Proof of Theorem~\ref{thm:vegtelen}]
The examples $(K_1(n), K_2(n))$ found above verify the theorem: the
same proof as the one given in Theorem~\ref{thm:am} applies and shows
that the $4$-manifolds corresponding to a pair $(K_1(n), K_2(n))$ are
homeomorphic but non-diffeomorphic, and by Proposition~\ref{prop:veg}
for $n\geq 0$ these examples are pairwise distinct.  By the proof of
Theorem~\ref{generatoros}, we see that for each $n \geq 0$ the smooth
$4$-manifolds corresponding to $K_1(n)$ and $K_2(n)$ are distinguished
by the $\sg$ invariants given in Definition~\ref{sginvdef}.  (The
proof of Theorem~\ref{generatoros} is given in the next Section.)
\end{proof}

\section{Maps on $4$-manifolds}
\label{sec:three}

One of the main ingredients of our arguments below is derived from a
construction of Saeki. This construction (under some specific
restrictions on the topology of the $4$-manifold) provides stable maps
on $4$-manifolds with strong control on their singular sets.  (In
fact, in our applications we will only use the existence part of the
equivalence.)

\begin{thm}(\cite[Theorem 3.1]{Sa03})
\label{thm:saeki}
Suppose that $X$ is a closed, oriented, connected, smooth $4$-manifold
with embedded nonempty (not necessarily connected) surfaces $F=F_0\cup
F_1$. There exists a fold map $f\colon X \to \R ^3$ with $F_0$ the definite and 
$F_1$ the indefinite fold singular locus if and only if
\begin{enumerate}[\rm (1)]
\item for the Euler characteristics $\chi (X)=\chi (F_0)-\chi (F_1)$ holds,
\item the Poincar\'e dual of the class $[F]$ represented by the surface $F$
(in mod $2$ homology) coincides with $w_2(X)$,
\item $F_0$ is orientable,
\item the self-intersection of every component  of $F_1$ is zero, and
\item the self-intersection of $F_0$ is equal to $3\sigma (X)$, where
  $\sigma (X) $ is the signature of $X$. \qed
\end{enumerate} 
\end{thm}

Let $K$ be a knot in $S^3$ and denote by $X=X_K$ the $4$-manifold
obtained by gluing a $2$-handle to $D^4$ determined by $K$ and framing
$-1$.  The Euler characteristic $\chi (X)$ of $X$ is equal to $2$.  An
embedded orientable surface $S$ coming from the core of the $2$-handle
and a surface in $D^4$ bounding $K$ represents the generator of
$H_2(X; \Z ) = \Z$. It follows that $S \cdot S = -1$.  If $K$ is
slice, then $S$ can be chosen to be a sphere.
 
A fold map is called a {\it
  definite fold map} if it has only definite fold singularities\footnote{Another terminology
  for a definite fold map is {\it special generic map} or {\it submersion with definite folds}.}.

\begin{prop}\label{Xenfold} 
  Let $X$ be a $4$-manifold given by attaching a single $2$-handle to
  $D^4$.  There is a definite fold map $f \co X \to \R^3$ such that
  $S$ is equal to the singular set of $f$.
\end{prop}
\begin{proof}
  Double $X$ along its boundary. It is easy to see that the resulting
  closed $4$-manifold is diffeomorphic to $\CP^2 \# \bCPk$,
  cf.\ \cite{GS99}.  Apply \cite[Theorem 3.1]{Sa03} of Saeki quoted in
  Theorem~\ref{thm:saeki} as follows. Let $F_0$ be $S \cup \bar S$ in
  $\CP^2 \# \bCPk$ and let $F_1$ be a standardly embedded surface of
  genus $= 1 + 2g(S)$ in a small local chart in the second copy of $X$
  such that $(S \cup \bar S) \cap F_1 = \emptyset$. Then conditions
  (1)-(5) of \cite[Theorem 3.1]{Sa03} are satisfied: $\chi(\CP^2 \#
  \bCPk) = \chi(S \cup \bar S) - \chi(F_1)=4$, the Poincar\'e dual of
  the homology class $[S \cup \bar S \cup F_1]$ is characteristic,
  hence reduces to $w_2(\CP^2 \# \bCPk)$ mod 2, $S \cup \bar S$ is
  orientable, $F_1 \cdot F_1= 0$ and $S \cdot S + \bar S \cdot \bar S
  = 0$.  Hence there is a fold map $f' \co \CP^2 \# \bCPk \to \R^3$
  such that $S$ is a component of the definite fold singular set.
  Thus restricting $f'$ to $X$ gives a definite fold map $f \co X \to
  \R^3$ such that $S$ is equal to the singular set of $f$.
\end{proof}
This construction provides the essential tool to prove
Theorem~\ref{generatoros}.
\begin{proof}[Proof of Theorem~\ref{generatoros}]
  Let $X_j=X_{K_j(n)}$ for some $n\in \bfn $ and $j\in \{ 1, 2\}$,
  where $K_j(n)$ are the knots found in the proof of
  Theorem~\ref{thm:vegtelen}.  By Proposition~\ref{Xenfold} both $X_1$
  and $X_2$ have definite fold maps into $\R^3$ such that the singular
  set components represent a generator for the second homology group
  $H_2$, thus $\sg^{1}(X_i, A) \neq \infty$, $i=1,2$.  We have
  that $\sg^{1}(X_1, A) = 0$ since $X_1$ is the manifold obtained
  by handle attachment along a slice knot.  On the other hand, as the
  proof of Theorem~\ref{thm:am} showed, an embedded sphere cannot
  represent a generator of $H_2(X_2; \Z )$ hence $\sg^{1}(X_2, A)
  > 0$. 
  Finally, clearly each of two disjoint surfaces cannot represent a generator, hence
  we get the result for $\sg^{k}$, where $k \geq 2$.
  \end{proof}

\subsection*{Maps on closed $4$-manifolds}
With a little bit of additional work, we can prove a similar statement
for closed $4$-manifolds, eventually leading us to the proof of
Theorem~\ref{minuszegyes}.

\begin{proof}[Proof of Theorem~\ref{minuszegyes}]
  Let $K_1=K_1(n)$ and $K_2=K_2(n)$ be one of the pairs of examples
  provided by Theorem~\ref{thm:vegtelen}. The property of $K_2$ being
  smoothly isotopic to a Legendrian knot with vanishing
  Thurston-Bennequin invariant implies that the
  $4$-manifold-with-boundary $X_2$ we get by attaching a
  $4$-dimensional $2$-handle along $K_2$ with framing $-1$ admits a
  Stein structure. By a result of Lisca-Mati\'c \cite{LM} the
  $4$-manifold $X_2$ therefore embeds into a minimal complex surface
  $Z$ of general type (which we can always assume to have
  $b_2^+>1$). Since $X_2$ has odd intersection form, it follows that
  the intersection form $Q_Z$ of $Z$ is also odd. Indeed, we can also
  assume that the intersection form of $Z-{\rm {int\ }} X_2$ is also
  odd.  Therefore $Q_Z$ can be written in an appropriate basis $B_Z =
  \{ e, a, b, f_1, \ldots , f_n, g_1, \ldots , g_m \}$ of $H_2(Z;
  {\mathbb {Z}})/Torsion$ as
\[
\langle -1 \rangle \oplus H\oplus n \langle -1 \rangle  
\oplus m \langle 1 \rangle  
\]
for some $n, m$, where the first summand (generated by $e$)
corresponds to the generator of $H_2 (X_2; \Z )$ and $H$ denotes a
hyperbolic pair with basis elements $a,b$.  Blow up $Z$ $j$-times
(with $j=0,1,2$ or $3$) in order to achieve that the resulting complex
surface $W$ has signature $\sigma (W)$ divisible by $4$: $\sigma
(W)=4k$. The basis elements $h_1, \ldots , h_j \in B_W = B_Z \cup \{
h_1, \ldots , h_j \}$ correspond to the exceptional divisors of the
(possible) blow-ups. Consider now the homology class $$\Sigma = e+
2\cdot a+2k\cdot b +\sum f_i +\sum g_j +\sum h_l.$$ Since by
definition $-1-n+m-j=\sigma (W)=4k$, it is easy to see that

\begin{lem}\label{szigmakar}
  The homology class $\Sigma $ is a characteristic element in the
  sense that the Poincar\'e dual $PD(\Sigma)$ reduced mod $2$ is equal
  to the second Stiefel-Whitney class $w_2(W)$, and the
  self-intersection of $\Sigma$ is equal to $-1+8k-n+m-j=3\sigma (W)$.
  \hfill $\Box$
\end{lem}
As any second homology class of a smooth $4$-manifold, the class
$\Sigma$ can be represented by a (not necessarily connected) oriented
surface $F_0\subset W$. Indeed, we can assume that the part $\sum h_l$
of $\Sigma $ is represented by $j$ disjoint embedded spheres of
self-intersection $-1$. Notice also that $W$ does not contain $j+1$
disjoint $(-1)$-spheres: since $Z$ has two Seiberg-Witten basic
classes $\pm c_1(Z)$ with $c_1^2(Z)>0$, by the blow-up formula (and
since it is of simple type) $W$ has $2^{j+1}$ basic classes. If $W$
had $(j+1)$ disjoint $(-1)$-spheres, then it could be written as
$W=Y\#_{j+1}{\overline {\mathbb {CP}}}^2$, hence by the blow-up
formula again $Y$ has a unique basic class, which is therefore equal
to $0$, implying that $c_1^2(Z)=-1$, a contradiction.  Since $b_2^+(W)
> 1$, there is no homologically non-trivial embedded sphere in the
complex surface $W$ with non-negative self-intersection.

Now let $X_1$ denote the $4$-manifold-with-boundary we get by
attaching a $4$-dimensional $2$-handle to $D^4$ along $K_1$. Since
$\partial X_1$ is diffeomorphic to $\partial X_2$, we can consider the
smooth $4$-manifold $V=X_1\cup (W-X_2)$. Notice that it is
homeomorphic to $W$ (since $X_1$ is homeomorphic to $X_2$). Consider
the homology class $\Sigma '\in H_2(V;{\mathbb {Z}})$ corresponding to
$\Sigma \in H_2 (W; \Z)$.  It can be represented by an orientable
embedded surface $F'_0$ which has $j+1$ disjoint spherical components
(all with self-intersection $(-1)$) and some further components. This
follows from the fact that the $j$ exceptional divisors of $W-X_2$ can
be represented by such spheres in $W-X_2=V-X_1$, and in addition, the
generator of $H_2(X_1; {\mathbb {Z}})$ also can be represented by an
embedded sphere (of self-intersection $(-1)$), since $K_1$ is a slice
knot.

Note that (since the signature $\sigma (W) $ is divisible by $4$)
the Euler characteristics 
$\chi(W)$ and $\chi(V)$ are even.  Let $F_1$ be a closed, orientable
surface embedded in $W$ such that $F_0 \cap F_1 = \emptyset$,
$\chi(F_1) = \chi(F_0) - \chi(W)$, $[F_1] = 0$ and $F^r_1 \cdot F^r_1
= 0$ for each connected component $F^r_1$ of $F_1$. (For example,
$F_1$ can be chosen to be standardly embedded in a local coordinate
chart of $W$.)  Similarly, let $F_1'$ be a closed, orientable surface
in $V$ such that $F_0' \cap F_1' = \emptyset$, $\chi(F_1') =
\chi(F_1') - \chi(V)$, $[F_1'] = 0$ and ${F'^r_1} \cdot F'^r_1 = 0$
for each connected component $F'^r_1$ of $F_1'$.

Then conditions (1)-(5) of \cite[Theorem 3.1]{Sa03} are satisfied for
$W$, $F_0$ and $F_1$: (1) is obvious by the choice of $F_1$, (2)
follows from Lemma~\ref{szigmakar}, (3) and (4) are obvious and (5)
follows from Lemma~\ref{szigmakar} as well.  Similarly, conditions
(1)-(5) of \cite[Theorem 3.1]{Sa03} are also satisfied for $V$, $F_0'$
and $F_1'$.  So there exist fold maps on $W$ and $V$ such that their
singular sets are the surfaces $F_0 \cup F_1$ and $F_0' \cup F_1'$,
respectively.  By construction, $V$ contains $j+1$ disjoint
$(-1)$-spheres, hence by Theorem~\ref{thm:saeki}, with the choice $k=j+1$ we
have $\sg^k (V, A)=0$.  For $W$ the argument following
Lemma~\ref{szigmakar} shows that with the same choice of $k$ we have
$\sg ^k (W, A)>0.$ 
It is easy to see 
that for any $1 \leq l \leq j$ we have $\sg^l (V, A)= \sg^l (W, A)=0$.
This completes the proof of
Theorem~\ref{minuszegyes}.
\end{proof}

\section{Stable maps and defects}
\label{sec:defect}
Let us start by recalling the notions of total defect, canonical
framing and stable framing \cite{KiMe}.  A \emph{stable framing} of a
$3$-manifold $M$ is a homotopy class of a trivialization (i.e.\ a
maximum number of linearly independent sections) of the trivial vector
bundle $TM \+ \ep^1$.  The degree $d(\phi)$ of a stable framing $\phi$
is the degree of the map $\nu \co M \to S^3$, where $\nu$ is the
framing of $\ep^1$.  The \emph{Hirzebruch defect} $h(\phi)$ of $\phi$
is defined to be $p_1(X, \phi) - 3\si(X)$, where $X$ is a compact
oriented $4$-manifold bounded by $M$.  The \emph{total defect}
$H(\phi)$ of $\phi$ is the pair $(d(\phi), h(\phi))$ and $H \co
\mathbb F_{\mathfrak s} \to \Z \+ \Z$ is an embedding of the set of
homotopy classes of stable framings $\mathbb F_{\mathfrak s}$
extending a fixed spin structure $\mathfrak s$ on $M$.  Finally, a
stable framing $\phi$ is \emph{canonical} for a spin structure
$\mathfrak s$ if it is compatible with $\mathfrak s$ and $|H(\phi)|
\leq |H(\psi)|$ for any stable framing $\psi$ which is compatible with
$\mathfrak s$.  It also follows that the invariant $2|d| + |h| \co
\mathbb F_{\mathfrak s} \to \N$ takes its minimum on a canonical
$\phi$.  A canonical framing may be not unique.  For details see
\cite{KiMe}.

A smooth map $f \co X \to \R^3$ of a $4$-manifold without
singularities near the boundary induces a stable framing $\phi$ on the
complement of a neighborhood $N(\Si)$ of the singular set $\Si$ such
that $\phi$ also gives a stable framing $\phi_{\del X}$ on the
boundary $3$-manifold.  We get $\phi = (\csi_0, \csi_1, \csi_2,
\csi_3)$ by taking the $1$-dimensional kernel $\csi_0$ of $df$ in the
tangent space of $X-N(\Si)$ and by pulling back the standard $1$-forms
$dx_1, dx_2, dx_3$ in $\R^3$ via the differential of the submersion
$f|_{X-N(\Si)}$.  Then a chosen Riemannian metric on $X-N(\Si)$ gives
$\csi_i$ as the dual to $dx_i$, $i=1,2,3$.

\begin{lem}
  Let $X$ be a compact oriented $4$-manifold with boundary
  and $\Si^0=
  \bigcup_{i=1}^{u} \Si_i$ and $\Si^1= \bigcup_{i=1}^{v} \Si_{u+i}$
  unions of closed, oriented, connected, non-empty disjoint surfaces
  embedded in $X$.  Assume $\Si^0 \cup \Si^1$ is disjoint from a
  neighborhood of $\del X$.  If there exists a fold map $X \to \R^3$
  with $\Si^0$ and $\Si^1$ as definite and indefinite fold singular
  sets, respectively, then
\begin{enumerate}[\rm (1)]
\item
the Hirzebruch defect $h(\phi_{\del X}) = \Si^0 \cdot \Si^0 -3\si(X)$,
\item
$\Si^1 \cdot \Si^1 = 0$,
\item
the Poincar\'e dual to the mod $2$ homology class $[\Si^0 \cup \Si^1]$ is  $w_2(X)$,
\item
the degree $d(\phi_{\del X}) = \chi(X) - \chi(\Si^0) + \chi(\Si^1)$ and
\item
$\phi_{\del X}$ is compatible with a spin structure on the complement $X- N(\Si^0 \cup \Si^1)$
of a tubular neighborhood of $\Si^0 \cup \Si^1$.
\end{enumerate}
\end{lem}

\begin{rem}
Our proof implies that if there exists such a fold map but 
$\Si^1 = \emptyset$, then
(1)-(5) still hold (if we define $\emptyset \cdot \emptyset = 0$ and 
$\chi(\emptyset) = 0$).
\end{rem}
\begin{proof}
Suppose that there exists such a fold map $f \co X \to \R^3$. 

(2) holds because the
normal disk bundle of $\Si^1$ is trivial, since the symmetry group of
the indefinite fold singularity germ $(x, y) \mapsto x^2 - y^2$ can be
reduced to a finite $2$-primary group.  

(3) holds because the map $f$
restricted to $X - (\Si^0 \cup \Si^1)$ is a submersion into $\R^3$
hence the tangent bundle of $X - (\Si^0 \cup \Si^1)$ has a framing.

For (5)
let $\phi$ denote the induced framing on $X - N(\Si)$.  Since the
Poincar\'e dual $PD[\Si^0 \cup \Si^1] \equiv w_2(X)$ (mod 2), $\phi$ is compatible
with a spin structure on $X - N(\Si)$ and this proves (5). 

For (4),
 we know by
\cite[Lemma~2.3 (b)]{KiMe} that $d(\phi_{\del (X - N(\Si))}) = \chi( X
- N(\Si))$ since
the stable framing $\phi_{\del (X - N(\Si))}$ on $\del (X - N(\Si))$
given by $f$ extends to a framing on $X- N(\Si)$.
We have $d(\phi_{\del (X - N(\Si))}) = d(\phi_{\del X}) -2
\chi(\Si^1)$ by \cite[Lemma~3.2]{Sa03}.  Hence we have $d(\phi_{\del
  X}) = 2 \chi(\Si^1) + \chi( X - N(\Si)) = \chi(\Si^1) + \chi(X) -
\chi(\Si^0)$, which proves (4). 

For (1), by
\cite[Lemma~2.3 (b)]{KiMe}  we have $p_1(X - N(\Si), \phi_{\del(X - N(\Si))}) = 0$.
Also we have that $\sum_{j} h(\phi_j)
= p_1(X - N(\Si), \phi) - 3\si(X - N(\Si))$, where $j$ runs over the
boundary components of $X - N(\Si)$ and $\phi_j$ denotes the
corresponding stable framing on that boundary component.  Hence
$h(\phi_{\del X}) = -3\si(X - N(\Si)) - \sum_{i=1}^{u+v} h(\phi_i)$,
where $\phi_1, \ldots, \phi_{u+v}$ are the stable framings on $\del
N(\Si_i)$, $i=1,\ldots, u+v$, respectively.  From \cite[Proof of
  Theorem~3.1 and Lemma~3.4]{Sa03}, we know that $h(\phi_i) = -\Si_i
\cdot \Si_i + 3\sgn(\Si_i \cdot \Si_i)$ if $\Si_i$ is a definite fold
component of $\Si$, otherwise $h(\phi_i) = 0$.  Note that $\Si_i \cdot
\Si_i = 0$ for indefinite fold singular set components.  Thus
$h(\phi_{\del X}) = -3\si(X - N(\Si)) + \Si \cdot \Si -
3\sum_{i=1}^{u+v} \sgn(\Si_i \cdot \Si_i)$ and
$$h(\phi_{\del X}) =  \Si \cdot \Si -3\si(X),$$
proving (1).
\end{proof}

Let $X$ be the $4$-manifold obtained by attaching a $2$-handle to
  $D^4$ along a $p$-framed knot in $S^3$, where $p \in \Z$.
  We can express the total
defect of the induced stable framing $\phi_{\del X}$ as follows.
\begin{prop}\label{foldtotaldefect}
    Let $f \co X \to \R^3$ be a fold map with $\Si^0 =
    \bigcup_{i=1}^{u} \Si_i$ and $\Si^1 = \bigcup_{i=1}^{v} \Si_{u+i}$
    as non-empty definite and indefinite fold singular sets,
    respectively, both consisting of closed connected orientable
    surfaces.  Then
\begin{enumerate}[\rm(1)]
\item the total defect $H(\phi_{\del X}) = (\chi(\Si^1) + 2 -
  \chi(\Si^0), -3p + p \sum_{i=1}^u k_i^2)$, where each component $\Si_i$ of
  the singular set represents $\pm k_i$ times the generator of
  $H_2(X;\Z)$, and 
\item if $\del X$ is an integral homology sphere, then $H(\phi_{\del
  X})$ is of the form $(2s, 4r+2)$ or of $(2s, 4r)$ for some $r, s \in
  \Z$.
\end{enumerate}
\end{prop}
\begin{proof}
We get that $h(\phi_{\del X}) = p \sum_i k_i^2 -3p$, where $i$ runs
over the definite fold components of $\Si$.  
 We also have $d(\phi_{\del X}) =
\chi(\Si^1) + 2 - \chi(\Si^0)$.  
This gives that if all the singular set components are orientable, then $d(\phi_{\del X})$ is even.  
The
total defect $H(\phi_{\del X})$ is equal to
\[
(\chi(\Si^1) + 2 - \chi(\Si^0), -3p +p \sum_{i=1}^u k_i^2).
\] 
If $\del X$ is a homology sphere, then by \cite[Theorem~2.6]{KiMe} we
obtain that $H(\phi_{\del X})$ should be in the coset $\Lambda_0 +
(0,k)$, where $k=0$ or $k=2$, and $\Lambda_0$ is the subgroup of $\Z
\+ \Z$ generated by $(0, 4)$ and $(-1, 2)$. Thus $h(\phi_{\del X})$ is
of the form $4r+2$ for $r,s \in \Z$ or of the form $4r$ for $r \in \Z$
(this depends on the $\mu$-invariant of $\del X$, see \cite{KiMe}).
\end{proof}

Let $X$ be the $4$-manifold obtained by attaching a $2$-handle to
  $D^4$ along a $(-1)$-framed
   knot in $S^3$.

\begin{prop}\label{Xencandeffold}
  There is a fold map $f \co X \to \R^3$ such that the total defect of
  the stable framing induced by $f$ on $\del X$ is canonical.
\end{prop}
\begin{proof}
  Double $X$ along its boundary. As before, the resulting closed
  $4$-manifold is diffeomorphic to $\CP^2 \# \bCPk$ (cf.\ \cite{GS99}).
  Let $S$ be an embedded orientable surface in $X$ coming from the
  core of the $2$-handle and a surface in $D^4$ bounding $K$ and hence
  representing the generator of $H_2(X; \Z ) = \Z$.
  
  Apply \cite[Theorem 3.1]{Sa03} as follows. Let $F_0$ be $S
  \cup \bar S$ in $\CP^2 \# \bCPk$ and $F_1$ be a
  union of two disjoint copies of a small closed orientable surface of Euler
  characteristic $\chi(S) - 2$  such that $F_1$ is
  null-homologous (and so its self-intersection is equal to $0$).  One
  component of $F_1$ is embedded into $X$, the other component into
  $\bar X$ and suppose
  $(S \cup \bar S) \cap F_1 = \emptyset$.
    Then conditions (1)-(5) of \cite[Theorem 3.1]{Sa03} are
  satisfied, so there exists a fold map on $\CP^2 \# \bCPk$ with $F_0
  \cup F_1$ as singular set.  Restricting this map to $X$, we get a
  fold map such that (by Proposition~\ref{foldtotaldefect}) the total
  defect of the induced stable framing on $\del X$ is equal to $(0,
  2)$.  
  Hence this
  stable framing is canonical.
\end{proof}

Note that if $S$ is a sphere, then a similar construction to Proposition~\ref{Xenfold} gives us
a definite fold map with the claimed property.

Now we are in the position to prove Theorem~\ref{kanonikus}.
\begin{proof}[Proof of Theorem~\ref{kanonikus}]
Once again, let $(K_1, K_2)$ be a pair of knots provided by
Theorem~\ref{thm:vegtelen}, and let $X_j=D^4_{-1}(K_j)$ be the
$2$-handlebody we get by attaching a single $2$-handle to $D^4$ along
$K_j$ with framing $(-1)$.  Clearly both $X_1$ and $X_2$ have fold
maps as given in the proof of Proposition~\ref{Xencandeffold}. These
stable maps satisfy property $\mathcal A$, so $\sg(X_i, \mathcal A) <
\infty$, $i=1,2$.  Obviously $\sg(X_1, \mathcal A) =0$.  If $X_2$ has
a fold map giving $\sg(X_2, \mathcal A) = 0$, then all the singular
set components are spheres and by Proposition~\ref{foldtotaldefect}
the total defect $H(\phi_{\del X_{2}}) = (2s, 3 - \sum_i k_i^2)$ for
some $s \in \Z$.  Since by assumption $\phi_{\del X_{2}}$ is canonical
and $\del X_{2}$ is a homology sphere, by \cite{KiMe} and
Proposition~\ref{foldtotaldefect} its total defect should be equal to
$(0, \pm 2)$ or $(0,0)$.  Hence $\sum_i k_i^2$ is equal to $1$, $3$ or
$5$, which implies that some $k_i = \pm 1$, which is impossible for
$X_2$ since the generator of its second homology cannot be represented
by a smoothly embedded sphere.
  \end{proof}

\begin{rem}
It is not difficult to obtain results similar to those of Sections 3 and 4
in the case of $\mathcal M$ being the set of all the definite fold maps
and $\mathcal S$ being the one element set of the definite fold singularity.
However, a $4$-manifold $X$ typically does not admit any definite fold map into $\R^3$, cf.\ \cite{SS99};
 in this case $\sg^{k}_{\mathcal M, \mathcal S}(X, A) = \infty$.
\end{rem}

  \section{Appendix: the calculation of the Conway and Jones
    polynomials}
\label{sec:appendix}
  In this appendix we give the details of the computation of the
  Ohtsuki invariants of the $3$-manifolds $S^3_{-1}(K_j(n))$ from
  Section~\ref{sec:two}.  For simplicity let $L_n$ denote $K_2(n)$,
  the knot we get from the knot $K_2$ of the lower part of
  Figure~\ref{elsolink_lefujas} after inserting the module of $n$ full
  twists in the box.

\begin{rem}
  From the knot tables we get that $L_0$ is the $5_2$ knot, $L_1 =
  9_{45}$ and $L_2 = 11n_{63}$.
\end{rem}

We define the Conway and Jones polynomials of oriented links using the
conventions (and, in particular, the skein relations and
normalizations) as they are given in \cite[Page~299]{LW99}.

\begin{lem}\label{lem:conwayszam}
The Conway polynomial of $L_n$ is equal to 
\[
\nabla (L_n)=1+2z^2-nz^4.
\]
In particular, the coeffient of the $z^4$-term is $-n$.
\end{lem}
\begin{proof}
Recall that the Conway polynomial of an oriented knot/link satisfies the 
skein relation
\[
\nabla (K_+)-\nabla (K_-)=-z\cdot \nabla (K_0)
\]
and is normalized as $\nabla (U)=1$, where $U$ denotes the unknot, and
$K_+, K_-, K_0$ admit projections identical away from a crossing,
where $K_+$ has a positive, $K_-$ a negative crossing and $K_0$ is the
oriented resolution.

The skein relation applied to any of the crossing in the module of $L_n$
shows that 
\[
\nabla (L_n)-\nabla (L_{n-1})=-z\cdot \nabla (J_0),
\]
where $J_0$ is the $2$-component link we get by replacing the module with 
Figure~\ref{fig:IO}. 
\begin{figure}[ht]
\begin{center}
\epsfig{file=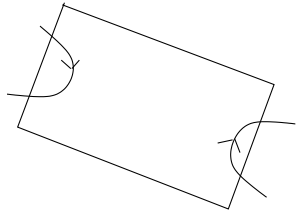, height=2.8cm}
\end{center}
\caption{{\bf The module giving the link $J_0$.} After inserting this
  module, we get a two-component link, one component being the right
  handed trefoil knot, the other one the unknot. The two components
  link geometrically twice, but with vanishing linking number.}
\label{fig:IO}
\end{figure}
This identity shows that $\nabla (L_n)= \nabla (L_0)-nz\cdot \nabla
(J_0)$. Further repeated application of the skein relation computes
$\nabla (L_0)=1+2z^2$.

The link $J_0$ has a trefoil component and an unknot component
(linking it twice, with linking number zero). The repeated application
of the skein relation shows that $\nabla (J_0)=z^3$. This computation
then proves the lemma.
\end{proof}

\begin{rem}\label{rem:azonositas}
  It is not hard to see that $L_0$ is isotopic to the knot $5_2$ in
  the usual knot tables, while the $2$-component link $J_0$ can be
  identified with the link $L7n2$ in Thistlethwaite's Link Table.
\end{rem}

\begin{lem}\label{lem:jonesszam}
  For $n \geq 1$ the Jones polynomial $V(L_n, t)$ of $L_n$ is
\[
(1 + t^{-2} + \cdots + (t^{-2})^{n-1}) \tilde V(t) + t^{-2n}V(L_0,t),
\]
where $\tilde V(t) = t^{-1}(t^{1/2}-t^{-1/2})V(J_0,t)= 2t^{-1} -
3t^{-2} + 3t^{-3} - 3t^{-4} +2t^{-5} -2t^{-6} + t^{-7}$ and $V(L_0, t)
= t^{-1} - t^{-2} + 2t^{-3} - t^{-4} +t^{-5} - t^{-6}$.
\end{lem}
\begin{proof}
We compute the Jones polynomial using the skein relation
\[
tV(K_+, t) - t^{-1}V(K_-,t) = (t^{1/2} - t^{-1/2})V(K_0),
\]
where the Jones polynomial of the unknot is defined to be $1$ and
$K_+$, $K_-$, $K_0$ are as in the previous lemma. As before, we apply
the skein relation to any of the crossing in the module of $L_n$. (The
two further links in the relation are $L_{n-1}$ and $J_0$ again.)
Therefore induction for $n \geq 0$ shows that
\[
V(L_n, t) = (1 + t^{-2} \cdots + (t^{-2})^{n-1})t^{-1}(t^{1/2} -
t^{-1/2}) V(J_0,t) + (t^{-2})^n V(L_0,t),
\]
By computing $V(J_0,t)$ and $V(L_0,t)$ using the same skein relation
we get the statement. (Cf.\ also Remark~\ref{rem:azonositas} regarding
the polynomials of $L_0$ and $J_0$.)
\end{proof}

\begin{lem}\label{lem:jonesszamok}
For the Jones polynomial of $L_n$ we have 
\[
\frac{ \del^2 V(L_n, e^{h})}{\del h^2}(0) = -12, \qquad
\frac{ \del^3 V(L_n, e^{h})}{\del h^3}(0) = 36n + 108.
\]
\end{lem}
\begin{proof}
Simple differentiation and substitution gives that
\begin{equation}\label{eq:szamitasok}
\tilde V(e^h)\vert _{h=0} = 0, \qquad \frac{ \del \tilde V(
  e^{h})}{\del h}\vert _{h=0} = 2,\qquad \frac{ \del^2 \tilde V(
  e^{h})}{\del h^2}\vert _{h=0} = -4, \qquad \frac{ \del^3 \tilde V(
  e^{h})}{\del h^3}\vert _{h=0} = -28.
\end{equation}
The identities 
\begin{multline*}
  \frac{ \del V(L_n, e^{h})}{\del h}(h) = (-2e^{-2h} -4e^{-4h} -
  \cdots - (2n-2)e^{(-2n+2)h})\tilde V(e^h) + \\
  + (1 + e^{-2h} + \cdots + e^{(-2n+2)h})\frac{ \del \tilde V(
    e^{h})}{\del h}(h) + \sum_{i=1}^6 (-1)^{i+1}(-2n-i)e^{(-2n-i)h} +
  (-2n-3)e^{(-2n-3)h},
\end{multline*}
\begin{multline*}
  \frac{ \del^2 V(L_n, e^{h})}{\del h^2}(h) = (2^2e^{-2h} + 4^2e^{-4h}
  + \cdots + (2n-2)^2e^{(-2n+2)h})\tilde V(e^h)+ \\ + 2(-2e^{-2h}
  -4e^{-4h} - \cdots - (2n-2)e^{(-2n+2)h})\frac{ \del \tilde V(
    e^{h})}{\del h}(h) + (1 + e^{-2h} + \cdots + e^{(-2n+2)h})\frac{
    \del^2 \tilde V( e^{h})}{\del h^2}(h) + \\ + \sum_{i=1}^6
  (-1)^{i+1}(-2n-i)^2e^{(-2n-i)h} + (-2n-3)^2e^{(-2n-3)h},
\end{multline*}
\begin{multline*}
  \frac{ \del^3 V(L_n, e^{h})}{\del h^3}(h) = (-2^3e^{-2h} -
  4^3e^{-4h} - \cdots - (2n-2)^3e^{(-2n+2)h})\tilde V(e^h)+ \\
  + 3(2^2e^{-2h} + 4^2e^{-4h} + \cdots + (2n-2)^2e^{(-2n+2)h})\frac{
    \del \tilde V( e^{h})}{\del h}(h) + \\ 3(-2e^{-2h} -4e^{-4h} - \cdots
  - (2n-2)e^{(-2n+2)h})\frac{ \del^2 \tilde V( e^{h})}{\del h^2}(h) +
  \\ + (1 + e^{-2h} + \cdots + e^{(-2n+2)h})\frac{ \del^3 \tilde V(
    e^{h})}{\del h^3}(h) + \sum_{i=1}^6
  (-1)^{i+1}(-2n-i)^3e^{(-2n-i)h} + (-2n-3)^3e^{(-2n-3)h},
\end{multline*}
together with the values determined in Equation~\eqref{eq:szamitasok}
now give
\begin{multline*}
\frac{ \del^2 V(L_n, e^{h})}{\del h^2}(0) =
-4n(n-1) -4n + \sum_{i=1}^6 (-1)^{i+1}(2n+i)^2 + (2n+3)^2 = \\=
-4n(n-1) -4n -12n -21 + (2n+3)^2 = -12.
\end{multline*}
This computation proves the first claim of the lemma, and it also
shows (by \cite{Mu94}) that the Casson invariant of $S^3_{-1}(L_n)$ is
equal to $-2$ (and, in particular, is independent of $n$). Furthermore
\begin{multline*}
\frac{ \del^3 V(L_n, e^{h})}{\del h^3}(0) = 
24(1^2 + \cdots + (n-1)^2) + 
24(1+\cdots + (n-1)) -28 n +
\sum_{i=1}^6 (-1)^{i}(2n+i)^3 - (2n+3)^3 =\\=
4(n-1)n(2n-1) + 12n(n-1) -28n + \sum_{i=1}^6 (-1)^{i}(2n+i)^3 - (2n+3)^3=36n+108,
\end{multline*}
verifying the second claim of the lemma.
\end{proof}

\end{document}